\documentclass[11pt]{article}

\usepackage{tikz}
\usepackage{graphicx}
\usepackage{verbatim}
\usepackage{url}
\usepackage{fullpage}
\usepackage{amssymb,amsfonts,amsmath,amsthm}
\usepackage[colorlinks=true,linkcolor=blue,citecolor=red]{hyperref}
\usepackage[capitalise]{cleveref}

\newcommand{\FormatAuthor}[3]{
\begin{tabular}{c}
#1 \\ {\small\texttt{#2}} \\ {\small #3}
\end{tabular}
}

\newcommand{\ov}{\overline}

\newcommand{\A}{{\mathcal A}}
\newcommand{\B}{{\mathcal B}}
\newcommand{\C}{{\mathcal C}}

\newcommand{\I}{{\mathcal I}}

\newcommand{\poly}{{\rm poly}}

\newcommand{\NP}{\ensuremath{\mathcal{NP}}}

\renewcommand{\P}{\ensuremath{\mathcal{P}}}
\newcommand{\N}{{\mathbb N}}
\newcommand{\R}{{\mathbb R}}
\newcommand{\E}{{\mathbb E}}
\newcommand{\F}{{\mathcal F}}

\newcommand{\eps}{\epsilon}
\newcommand{\seq}{\subseteq}
\newcommand{\suq}{\supseteq}

\renewcommand{\int}{{\sf int}}

\newtheorem{theorem}{Theorem}[section]
\newtheorem{corollary}[theorem]{Corollary}
\newtheorem{proposition}[theorem]{Proposition}

\newtheorem{lemma}[theorem]{Lemma}
\newtheorem{claim}[theorem]{Claim}
\newtheorem{observation}[theorem]{Observation}

\newtheorem{question}[theorem]{Question}

\newenvironment{proof-sketch}{\noindent{\bf Sketch of Proof}\hspace*{1em}}{\qed\bigskip}
\newenvironment{proof-idea}{\noindent{\bf Proof Idea}\hspace*{1em}}{\qed\bigskip}
\newenvironment{proof-of-lemma}[1]{\noindent{\bf Proof of Lemma #1}\hspace*{1em}}{\qed\bigskip}
\newenvironment{proof-of-claim}[1]{\noindent{\bf Proof of Claim #1}\hspace*{1em}}{\qed\bigskip}
\newenvironment{proof-of-thm}[1]{\noindent{\bf Proof of Theorem #1}\hspace*{1em}}{\qed\bigskip}
\newenvironment{proof-attempt}{\noindent{\bf Proof Attempt.}\hspace*{1em}}{\qed\bigskip}

\newenvironment{remark}{\noindent{\it Remark.}}{\bigskip}

\renewcommand{\leq}{\leqslant}

\renewcommand{\geq}{\geqslant}
\renewcommand{\epsilon}{\varepsilon}
\newcommand{\uncut}{{\rm uncut}}
\newcommand{\cut}{{\rm cut}}

\title{On Coloring Random Subgraphs of a Fixed Graph}

\author{
\begin{tabular}[h!]{cc}
   \FormatAuthor{Igor Shinkar}{igors@berkeley.edu}{UC Berkeley}
\end{tabular}
}

\begin{document}

\maketitle
\thispagestyle{empty}

\begin{abstract}
Given an arbitrary graph $G$ we study the chromatic number of a random subgraph $G_{1/2}$
obtained from $G$ by removing each edge independently with probability $1/2$.
Studying $\chi(G_{1/2})$ has been suggested by Bukh~\cite{Bukh}, who asked whether
$\E[\chi(G_{1/2})] \geq \Omega( \chi(G)/\log(\chi(G)))$ holds for all graphs $G$.
In this paper we show that for any graph $G$ with chromatic number $k = \chi(G)$
and for all $d \leq k^{1/3}$ it holds that $\Pr[\chi(G_{1/2}) \leq d] < \exp \left(- \Omega\left(\frac{k(k-d^3)}{d^3}\right)\right)$.
In particular, $\Pr[G_{1/2} \text{ is bipartite}] < \exp \left(- \Omega \left(k^2 \right)\right)$.
The later bound is tight up to a constant in $\Omega(\cdot)$, and is attained when $G$ is the complete graph on $k$ vertices.

As a technical lemma, that may be of independent interest, we prove that
if in \emph{any} $d^3$ coloring of the vertices of $G$ there are at least
$t$ monochromatic edges, then $\Pr[\chi(G_{1/2}) \leq d] < e^{- \Omega\left(t\right)}$.

We also prove that for any graph $G$ with chromatic number $k = \chi(G)$ and independence number $\alpha(G) \leq O(n/k)$
it holds that $\E[\chi(G_{1/2})] \geq \Omega \left( k/\log(k) \right)$.
This gives a positive answer to the question of Bukh for a large family of graphs.
\end{abstract}


\section{Introduction}\label{sec:intro}

Given a graph $G = (V,E)$ and a parameter $p \in (0,1)$,
let $G_p$ denote a subgraph of $G$ where each edge of $G$ appears in $G_p$ with independently with probability $p$.
In this paper we study the chromatic number of $G_{1/2}$ for an arbitrary graph $G$,
whose chromatic number is equal to some parameter $k$.
Clearly, since $G_{1/2}$ is a subgraph of $G$, it holds that $\chi(G_{1/2}) \leq \chi(G)$.
If $G$ is the $k$-clique, then this is the well studied Erd\"{o}s-R\'{e}nyi random graph model~\cite{ER60},
where is it known that $\chi(G_{1/2}) = \Theta(\frac{k}{\log(k)})$ with high probability (see, e.g.~\cite{BollobasBook}).
By monotonicity we also see that if $\chi(G)=k$ and $G$ contains a $k$-clique, then with high probability $\chi(G_{1/2}) \geq \Omega(\frac{k}{\log(k)})$.
It is also not difficult to come up with an example of a graph $G$ for which $\chi(G_{1/2}) = \chi(G)$ with high probability.
(For instance, let $G$ be the complete $k$-partite graph with $\exp(k)$ vertices in each part.)
Given the foregoing examples, it is natural to ask whether
$\chi(G_{1/2})$ must be large with high probability for \emph{any} graph $G$ whose chromatic number $\chi(G)$ is large.
Studying $\chi(G_{1/2})$ has been suggested by Bukh~\cite{Bukh}, who asked the following question.

\medskip
\begin{minipage}{0.9\textwidth}
    \center{Is there a constant $c>0$ such that $\E[\chi(G_{1/2})]>c \cdot \frac{\chi(G)}{\log \chi(G)}$ for all graphs $G$?}
\end{minipage}

\medskip
\noindent
Recently there has been some work generalizing the classical result on random graphs,
asking about properties of random subgraphs of fixed graphs satisfying certain properties~\cite{Borgs1,Borgs2,FK13,KLS15}.
For example, there have been several results studying the emergence of a giant component in $G_p$ when $G$ is an expander graph~\cite{ABS04,FKM04,KS13},
and relating it to the well studied Erd\"{o}s-R\'{e}nyi random model~\cite{ER60}.
In this work we study a problem of a similar flavor, trying to relate the chromatic random of $G_p$ to the chromatic number of a random graph in the Erd\"{o}s-R\'{e}nyi random model.

In a slightly different context, this problem is also motivated by a recent work of Bennett et~al.~\cite{BRS},
who asked about the computational complexity of $\NP$-complete problems,
whose inputs come from a certain semi-random model.
In particular, they showed that many natural $\NP$-complete problems, such as finding the chromatic number of a graph,
or deciding whether a graph contains a Hamiltonian path, remain $NP$-hard
even in the seemingly relaxed situation, where the inputs to the problem
come from random subgraphs of worst case instances.
In particular, they proved that if $\chi(G) = k$, then
$\Pr[\chi(G_{1/2}) < d)] < \poly(\frac{1}{k-d^3})$ for all $d < k^{1/3}$.

\subsection{On the distribution of $\chi(G_{1/2})$}\label{sec:problem}

The problem of lower bounding the expected chromatic number of $G_{1/2}$
over all graphs $G$ with $\chi(G) = k$ has been asked several times in the past
\cite{open-problem-garden, Bellairs16}.
Not being able to answer this particular question, it is natural to ask a related, more refined, question.
Namely, what can we say about the \emph{distribution} of $G_{1/2}$.
In particular, is it possible to compare the distribution of $\chi(G_{1/2})$ for an arbitrary $G$ with $\chi(G) = k$
with the distribution of $\chi(G(k,1/2))$, the chromatic number of the Erd\"{o}s-R\'{e}nyi graphs?
For example, is it true that for all graphs $G$ with $\chi(G) = k$ and for all $d < ck/\log(k)$
if holds that $\Pr[\chi(G_{1/2}) \leq d] < \Pr[\chi(G(k,1/2)) \leq d]$?
In general, the answer to this question is negative, as can be observed by letting $k=3$ and taking $G = C_n$ to be the odd length cycle of length $n$.
Indeed, in this case $\Pr[\chi(G_{1/2}) \leq 2] = 1 - 2^{-n}$, while $\Pr[\chi(G(k,1/2)) \leq 2] = 7/8$.
This example can be easily extended to any $k$ divisible by $3$ and $d = 2k/3$:
for all $n \in \N$ sufficiently large there exists a graph $G$ with $\chi(G) = k$ such that
$\Pr[\chi(G_{1/2}) \leq 2k/3] \geq 1 - 2^{-c_1 n}$, while $\Pr[\chi(G(k,1/2)) \leq 2k/3] \leq 1 - c_2$,
where $c_1,c_2 > 0$ are some constants that depends only on $k$ (and are independent of $n$); we omit the details.
Nonetheless, it is natural to ask whether a relaxed comparison is true.

\begin{question}\label{prob:domination}
  Is there a fixed polynomial $\poly(\cdot)$ such that
  for every graph $G$ with $\chi(G) = k$ and for all $d \leq k$
  it holds that
  \[
    \Pr[\chi(G_{1/2}) \leq d] < \poly(\Pr[\chi(G(k,1/2)) \leq d]),
  \]
  where $G(k,1/2)$ is the random Erd\"{o}s-R\'{e}nyi graph?
\end{question}

Note that the probability of the event $\chi(G(k,1/2) \leq d$ can be bounded explicitly as follows.

\begin{proposition}
In the Erd\"{o}s-R\'{e}nyi random model $G(k,1/2)$
for all $d < \frac{k}{2\log(k)}$ it holds that
\[
    \Pr[\chi(G(k,1/2)) \leq d] < e^{- \Omega(\frac{k(k - d\log(d))}{d})}.
\]
In particular, $\Pr[G_{1/2} \text{ is bipartite}] < e^{- \Omega(k^2)}$.
\end{proposition}

\begin{proof}
We claim first that for all $d < k/2$, and for any partition of the $k$ vertices into $d$ classes, the $k$-clique
has at least $k^2/4d$ edges whose both endpoints belong to same class.
Indeed, let $V_1 \cup \cdots \cup V_d$ be a partition of the $k$ vertices into $d$ classes.
Then, the number of edges whose both endpoints belong to same class
is $\sum_{i=1}^d {|V_i| \choose 2} \geq d {\sum |V_i|/d \choose 2} = d {k/d \choose 2} \geq k^2/4d$,
where the leftmost inequality is by Jensen's inequality, using the fact that $x \mapsto \frac{x\cdot(x-1)}{2}$ is a convex function.

Therefore, the probability of removing all these edges is $2^{-k^2/4d}$.
By taking union bound over all $d^k$ possible $d$-colorings we conclude that
with high probability all possible $d$-colorings are violated.
Specifically, the probability that there exists a legal $d$-colorings in $G_{1/2}$ is upper bounded by
\[
    \Pr[\chi(G(k,1/2)) \leq d] \leq d^k \cdot 2^{-k^2/4d} = 2^{-(\frac{k(k - 4 d\log(d))}{4d})},
\]
as required.
\end{proof}

Note that the same argument applies also for any graph $G$
whose number of vertices is not too large compared to its chromatic number,
and gives a positive answer to \cref{prob:domination}
for all $n$-vertex graphs $G$ with $\chi(G) = k$ and for all $d$ such that $d \log(d) < k^2/8n$.

\begin{proposition}\label{prop:union bound argument}
Let $G = (V,E)$ is an $n$-vertex graph, whose chromatic number is $\chi(G) = k$.
Then for $d \leq k$ it holds that
\[
    \Pr[\chi(G_{1/2}) \leq d] \leq d^n \cdot 2^{-k^2/4d} = 2^{-(k^2/4d - n \log(d))}.
\]
In particular, if $n \leq Ck$, then $\E[\chi(G_{1/2})] \geq \frac{k}{8C\log(k)}$.
\end{proposition}

\begin{proof}
Let $V_1 \cup \cdots \cup V_d$ be a partition of $V$ into $d$ classes.
Note that since $\chi(G) = k$ it follows that $\sum_{i=1}^d \chi(G[V_i]) \geq k$.
Observing that the number of edges induced by $V_i$ is at least ${\chi(G[V_i]) \choose 2}$
(see \cref{claim:chi vs |E|}) this implies that
$G$ has at least $\sum_{i=1}^d {\chi(G[V_i]) \choose 2} \geq d {\frac{1}{d}\sum_{i=1}^d \chi(G[V_i]) \choose 2} \geq \frac{1}{d} {k/d\choose 2} \geq k^2/4d$ edges whose both endpoints belong to same class,
where the leftmost inequality is again by Jensen's inequality

Therefore, the probability of removing all these edges is $2^{-k^2/4d}$.
By taking union bound over all $d^n$ possible $d$-colorings we conclude that
the probability that there exists a legal $d$-colorings in $G_{1/2}$ is upper bounded by
\[
    \Pr[\chi(G(k,1/2)) \leq d] \leq d^n \cdot 2^{-k^2/4d}.
\]
For the ``in particular'' part, note that if $n \leq Ck$ then $C \geq 1$,
and for $d = k/4C\log(k)$ it holds that
\[
    \Pr[\chi(G_{1/2}) \leq d]
    \leq 2^{-(k^2/4d - n \log(d))}
    \leq 2^{-(Ck \log(k) - Ck \log(d))}
    \leq 2^{-2Ck}
    < 0.5.
\]
Therefore,
\[
    \E[\chi(G_{1/2})] \geq d \cdot \Pr[\chi(G_{1/2}) \geq d] \geq \frac{k}{8C\log(k)},
\]
as required.
\end{proof}

Note that, in general, if the number of vertices in $G$ is much larger than $k$,
then the argument in the proof of \cref{prop:union bound argument}
will not work, since the union bound will not
suffice in order to bound the probability that \emph{all} possible coloring are violated.
As a concrete example, consider a graph that consists of a $k$-clique connected to
path of length $N \gg k$. Then, the na\"ive union bound will be over $d^N$ possible colorings,
although the $N$ vertices on the path are ``irrelevant'', and we would like to avoid ``paying''
union bound for them.

\subsection{Our results}\label{sec:result}
The main result in this paper gives a positive answer to the question when $d$ is a constant independent of $k$.

\begin{theorem}\label{thm:tiny chi}
    Let $G = (V,E)$ be an arbitrary graph with $\chi(G) = k$.
    \begin{enumerate}
    \item\label{item:bipartite}
        $\Pr[\text{$G_{1/2}$ is bipartite}] < \exp \left(- \Omega \left(k^2 \right)\right)$.
    \item\label{item:chi leq d}
        For all $d \leq k^{1/3}$ it holds that $\Pr[\chi(G_{1/2}) \leq d] < \exp \left(- \Omega\left(\frac{k(k-d^3)}{d^3}\right)\right)$.  \label{item:d < k^(1/3)}
    \item\label{item:martingale}
        For all $0 < \eps < 1/2$ it holds that $\Pr[\chi(G_{1/2}) \leq (1-\eps)\sqrt{k}] < \exp \left(- \Omega\left(\eps^2\sqrt{k}\right)\right)$.
    \end{enumerate}
\end{theorem}

Note that $\Pr[\text{$G_{1/2}$ is bipartite}] = \exp(-\Theta(k^2))$, and hence the first item
of the theorem is essentially tight, and gives a positive answer to \cref{prob:domination}.
Slightly more generally, for any $d$ it holds that
$\Pr[\chi(G(k,1/2)) \leq d] \geq \exp(-\Theta(k^2/d))$, and hence for constant $d$ (independent of $k$)
the second item of the theorem gives a positive answer to \cref{prob:domination}.

As part of the proof of \cref{thm:tiny chi} we prove the following technical lemma.
The lemma says that if $G$ is \emph{$t$-far from being $d$-colorable},
then $\Pr[\chi(G_{1/2}) \leq d'] \leq \exp(-\Omega(t))$ for some $d'$ that depends on $d$,
where $t$-far from being $d$-colorable means that \emph{every} $d$-coloring of the vertices
has at least $t$ monochromatic edges.

\begin{lemma}\label{lemma:t monochromatic edges missed}
    Let $G = (V,E)$ be an arbitrary graph, and let $d,t \in \N$ be parameters.
    If in every coloring of the vertices of $G$ with $d^3$ colors
    there are at least $t$ monochromatic edges,
    then $\Pr[\chi(G_{1/2}) \leq d] < \left(\frac{\sqrt{5}-1}{2}\right)^{t} < 0.62^t$.
\end{lemma}

\medskip

Next, we study $\chi(G_{1/2})$ for a special (rather large) class of graphs.
Note that if $G$ is an $n$-vertex graph with $\chi(G)=k$, then $G$ contains
an independent set of size $n/k$. In many cases the maximal independent
set of $G$ is within a multiplicative constant factor of $n/k$, i.e.,
$\alpha(G) \leq C \cdot \frac{n}{k}$ for some $C>1$ that is not too large.
For example, the random graph models $G(n,p)$ and $G(n,d)$ satisfy this property
with high probability for all $p > \frac{1}{n}$ and $d \geq 2$ (see, e.g.,~\cite{BollobasBook}).
For such graphs we prove the following theorem.

\begin{theorem}\label{thm:alpha bound}
    Let $G = (V,E)$ be a graph with $\alpha(G) \leq C \cdot \frac{n}{k}$ for some $C>1$.
    Then for all $d \leq \frac{k}{16 C \log(k)}$ it holds that
    \[
        \Pr[\chi(G_{p}) \leq d] \leq \Pr[\alpha(G_{p}) \geq n/d] \leq 2^{-\frac{p k n}{8 C d^2}}.
    \]
    In particular, for all $p > \frac{1}{k}$ it holds that
    \[
        \E[\chi(G_{p})] \geq \frac{pk}{32C \log(pk)}.
    \]
\end{theorem}

The following corollary follows immediately from \cref{thm:alpha bound}.
\begin{corollary}\label{cor:subgraph col}
    Let $G = (V,E)$ be a graph with $\chi(G) = k$,
    and suppose that $G$ contains a subgraph $G' = (V',E')$ with $V' \seq V$
    such that $\alpha(G') \leq C \frac{|V'|}{k}$.
    Then,
    \[
        \E[\chi(G_{1/2})] \geq \frac{k}{8C\log(k)}.
    \]
\end{corollary}

\paragraph{On Hadwiger number of a random subgraph of a fixed graph:}
Next, we discuss a related graph parameter, called the \emph{Hadwiger number} of a graph.
Hadwiger number of a graph $G$, denoted by $h(G)$, is the maximal $t \in \N$ such that $G$ contains $K_t$ as a minor.
Hadwiger's conjecture states that $h(G) \geq \chi(G)$ for all graphs $G$.
While the conjecture is open for general graphs, the inequality $h(G) \geq \chi(G)$ is known to hold for a random graph $G(n,1/2)$ with high probability.
Mader~\cite{Mader68} proved an approximate version of the conjecture, namely that $h(G) \geq \Omega \left(\frac{\chi(G)}{\log(k)} \right)$
for all graphs $G$.
Kostochka~\cite{Kostochka84} improved Mader's result, and showed that $h(G) \geq \Omega \left(\frac{\chi(G)}{\sqrt{\log(k)}} \right)$ for all graphs $G$.
Motivated by this line of research Adrian Vetta~\cite{Vetta} asked the following question.
What is $\min\{ \E[h(G_{1/2})] : G \text{ such that } \chi(G) = k \}$?
We note that an almost tight answer to this question follows almost immediately from Kostochka's work~\cite{Kostochka84}.

\begin{theorem}\label{thm:hadwiger}
    Let $G = (V,E)$ be an $n$-vertex graph with $\chi(G) = k$.
    Then,
    \[
        \Pr \left[h(G) \geq \Omega \left( \frac{k}{\sqrt{\log(k)}} \right) \right] \geq 1 - \exp \left(-\Omega( k^2) \right ).
    \]
    In particular, $\E[h(G_{1/2})] \geq \Omega \left( \frac{k}{\sqrt{\log(k)}} \right)$.
\end{theorem}
\begin{remark}
By the result of~\cite{BCE80} when $G$ is the $k$-clique we have
$\E[h(G_{1/2})] \leq O \left( \frac{k}{\sqrt{\log(k)}} \right)$,
and so, the bound in \cref{thm:hadwiger} is tight up to a multiplicative constant.
\end{remark}
\section{Preliminaries}\label{sec:prelim}

Let $G = (V,E)$ be an undirected graph.
An independent set in $G$ is a subset of the vertices that spans no edges.
The \emph{independence number} of $G$, denoted by $\alpha(G)$, is the largest size of an independent set in $G$.
A legal \emph{vertex coloring} of $G$ is an assignment of colors to each vertex in $V$ such that no two adjacent vertices
have the same color.
The chromatic number of $G$, denoted by $\chi(G)$, is the smallest number of colors
required to legally color $G$. Note that in any vertex coloring of $G$ each color class forms an
independent set, and hence $\alpha(G) \geq n/\chi(G)$.
\emph{Hadwiger number} of a graph $G$, denoted by $h(G)$, is the maximal $t \in \N$ such that $G$ contains $K_t$ as a minor.

For a subset of the vertices $A \seq V$ let $E(A)$ be the set of edges spanned by $A$,
i.e., $E(A) = \{ (u,v)\in E: u,v \in A \}$.
For two disjoint subsets of the vertices $A,B \seq V$ define $\cut(A,B) = \{(u,v) \in E : u \in A, v \in B\}$
to be the set of edges with one endpoint in $A$ and one endpoint in $B$.

\medskip

We will need the following easy claim saying that the number of edges in
a graph is at least quadratic in its chromatic number.

\begin{claim}\label{claim:chi vs |E|}
  Let $G = (V,E)$ be a graph with chromatic number $\chi(G) = k$.
  Then $|E| \geq {k \choose 2}$.
\end{claim}
\begin{proof}
  Let $V = C_1 \cup \dots \cup C_k$ be a partition of
  the vertices of $G$ into $k$ color classes.
  Note that there must be at least one edge between every two color classes,
  as otherwise, if there are no edges between $C_i$ and $C_j$, then
  $C_i \cup C_j$ is an independent set, and we can color them with the same color,
  which implies that $\chi(G) \leq k-1$. Therefore $|E| \geq {k \choose 2}$.
\end{proof}

We will also need a result from extremal set theory about $r$-wise $t$-intersecting families.
In order to explain the result we will need some notation. Let $X$ be a finite set,
and let $P(X) = \{F : F \seq X\}$ be the collection of all subsets of $X$.
A family of sets $\F \seq P(X)$ is said to be $r$-wise $t$-intersecting
if for every $F_1,\dots,F_r \in \F$ it holds that $|\F_1 \cap F_2 \cap \dots \cap F_r | \geq t$.
In the proof of \cref{thm:tiny chi} we will use the following theorem due to Frankl~\cite{Frankl}.
\begin{theorem}[Claim 9.2~\cite{Frankl}]\label{thm:Frankl}
  Let $\F \seq P(X)$ be a $3$-wise $t$-intersecting family.
  Then $|\F| < (\frac{\sqrt{5}-1}{2})^t \cdot 2^n$.
\end{theorem}

\section{Probability that $\chi(G_{1/2})$ is small}\label{sec:proofs}

Below we prove \cref{lemma:t monochromatic edges missed}, this will immediately imply
the first two items of \cref{thm:tiny chi}.

\begin{proof}[Proof of \cref{lemma:t monochromatic edges missed}]
    For a graph $G = (V,E)$ let $\A = \{\A_1, \dots, \A_d \}$ be a partition/coloring of the vertices of $G$ into $d$ color classes,
    i.e., the $\A_i$'s are pairwise disjoint and $V = \A_1 \cup \dots \cup \A_d$.
    For such a partition $\A$ define $\uncut(\A) = E(\A_1) \cup \dots \cup E(\A_d) \seq E$.
    to be the set of the edges of $G$ with both endpoints in some $\A_i$.
    Denote by $\P_d = \{\A = (\A_1, \dots, \A_d)\}$ the collection of all such partitions
    of the vertices into $d$ color classes.

    Let $G_{1/2} = (V, E_{1/2})$ be a random subgraph of $G$, and
    let $S = E \setminus E_{1/2}$ be a random subset of the edges that are in $G$ but not in $G_{1/2}$.
    Note that $S$ can be sampled by added each edge of $G$ to $S$ independently with probability $1/2$.
    Then
    \begin{equation*}
        \Pr[\chi(G_{1/2}) \leq d] = \Pr[\exists \A \in \P_d: \uncut(\A) \seq S].
    \end{equation*}
    Let $U$ be the monotone closure of $\{\uncut(\A) : \A \in \P_d\}$
    defined as
    \begin{equation}\label{eq:def of U}
        U = \{\uncut(\A) : \A \in \P_d\}^\uparrow = \{ S \seq E : \exists \A \in \P_d \mbox{ such that } \uncut(\A) \seq S  \}.
    \end{equation}
    This implies that
    \[
        \Pr[\chi(H) \leq d] = \Pr[\exists \A: S \suq \uncut(\A)] = \frac{|U|}{2^{|E|}},
    \]
    and hence, it is enough to show an upper bound on the size of $U$.
    In order to upper bound $|U|$ we use the assumption that every $d^3$ coloring of the vertices of $G$
    has at least $t$ monochromatic edges.
    \begin{claim}\label{claim:3-wise t-intersecting}
    Let $U$ be as in \cref{eq:def of U}.
    Then $U$ is a 3-wise $t$-intersecting family.
    \end{claim}
    \begin{proof}
        Note that since $U$ is a monotone closure of $\{\uncut(\A) : \A \in \P_d\}$
        it is enough to show that $\{\uncut(\A) : \A \in \P_d\}$ is a 3-wise $t$-intersecting family.
        Let $\A = \{\A_1, \dots, \A_d \}$, $\B = \{\B_1, \dots, \B_d \}$, and $\C = \{\C_1, \dots, \C_d\}$
        be three partitions of the vertices $\A,\B,\C \in \P_d$.
        We claim that $|\uncut(\A) \cap \uncut(\B) \cap \uncut(\C)| \geq t$.
        Indeed, let
        \[
            \I = \{\A_i \cap \B_j \cap C_k : 1 \leq i,j,k \leq d, \A_i \cap \B_j \cap C_k \neq \emptyset\}
        \]
        be a partition of the vertices into at most $d^3$ color classes.
        By the assumption that every $d^3$ coloring of the vertices of $G$
        has at least $t$ monochromatic edges it follows that there are at least $t$ edges
        $e \in E$ that are contained in some color class of $\I$,
        i.e., $e \in E(\A_i \cap \B_j \cap C_k)$ for some $1 \leq i,j,k \leq d$.
        By the containment $E(\A_i \cap \B_j \cap C_k) \seq E(\A_i) \cap E(\B_j) \cap E(C_k)$
        this implies that there are least $t$ edges $e \in E$ that belong to
        $\uncut(\A) \cap \uncut(\B) \cap \uncut(\C)$, and the claim follows.
    \end{proof}
        By applying \cref{thm:Frankl} we conclude that $|U| \leq (\frac{\sqrt{5}-1}{2})^t \cdot 2^{|E|}$,
        and hence
        \[
            \Pr[\chi(H) \leq d] = \Pr[\exists \A: S \suq \uncut(\A)] = \frac{|U|}{2^{|E|}} \leq (\frac{\sqrt{5}-1}{2})^t,
        \]
        as required.
\end{proof}

We are now ready to prove the first two items of \cref{thm:tiny chi}
\begin{proof}[Proof of \cref{thm:tiny chi} \cref{item:bipartite}]
    Let $G = (V,E)$ be a graph with $\chi(G) = k$,
    and let $V_1, \dots, V_8$ be partition $V$ into 8 color classes.
    Denote by $m_i$ the number of edges spanned by $V_i$, and denote by $k_i$ the chromatic number of $G[V_i]$.
    Then, $\sum_{i=1}^8 k_i \geq k$, and hence for some $1 \leq i^* \leq 8$ we must have $k_{i^*} \geq k/8$.
    On the other hand, by \cref{claim:chi vs |E|} we have
    $m_{i^*} \geq {k_{i^*} \choose 2} \geq {k/8 \choose 2} \geq \Omega(k^2)$.
    Therefore, any $8$-coloring of $G$ has at least $\Omega(k^2)$ monochromatic edges,
    and thus, by \cref{lemma:t monochromatic edges missed} it follows that
    $\Pr[\text{$G_{1/2}$ is bipartite}] = \Pr[\chi(G_{1/2}) \leq 2] < 2^{-\Omega(k^2)}$.
\end{proof}

\begin{proof}[Proof of \cref{thm:tiny chi} \cref{item:chi leq d}]
    Let $G = (V,E)$ be a graph with $\chi(G) = k$,
    and let $V_1, \dots, V_{d^3}$ be partition $V$ into $d^3$ color classes.
    Denote by $m_i$ the number of edges spanned by $V_i$, and denote by $k_i$ the chromatic number of $G[V_i]$.
    Then, $\sum_{i=1}^{d^3} k_i \geq k$, and hence for some $1 \leq i^* \leq d^3$ we must have $k_{i^*} \geq k/d^3$.
    This implies
    \[
         \sum_{i=1}^{d^3} m_i
         \stackrel{(*)}{\geq} \sum_{i=1}^{d^3}  {k_i \choose 2}
         \stackrel{(**)}{\geq} d^3 {\frac{1}{d^3}\sum_i k_i \choose 2}
         \geq d^3 {\frac{k}{d^3} \choose 2}
         = \frac{k(k-d^3)}{2d^3},
    \]
    where (*) is by \cref{claim:chi vs |E|},
    and (**) is by Jensen's inequality,
    using the fact that $x \mapsto \frac{x\cdot(x-1)}{2}$ is a convex function.
    Therefore, any $d^3$-coloring of $G$ has at least $\frac{k(k-d^3)}{2d^3}$ monochromatic edges,
    and thus, by \cref{lemma:t monochromatic edges missed} it follows that
    $\Pr[\chi(G_{1/2}) \leq d] < 2^{-\Omega(\frac{k(k-d^3)}{2d^3})}$.
\end{proof}

\begin{proof}[Proof of \cref{thm:tiny chi} \cref{item:martingale}]
We start with the following easy observation.
\begin{observation}\label{obs:sqrt(n)}
    If $\chi(G) = k$, then $\E[\chi(G)_{1/2}] \geq \sqrt{k}$.
\end{observation}

\begin{proof}
    Let $G = (V,E)$.
    Let $H = (V,E_H)$ be a subgraph of $G$ and let $\ov{H} = (V,E \setminus E_H)$ be the complement of $H$ in $G$.
    Note that $\chi(H) \cdot \chi(\ov{H}) \geq k$.
    Indeed, if $c_H:V \to [\chi(H)]$ is a coloring of $H$ and
    $c_{\ov H}:V \to [\chi(\ov H)]$ is a coloring of $\ov H$,
    then we can construct a coloring $c_G$ of $G$
    with at most $\chi(H) \cdot \chi(\ov{H})$ colors
    by letting $c_G(v) = (c_H(v),c_{\ov H}(v))$.
    To see that $c_G$ is indeed a legal coloring note that
    every edge $(u,v)$ in $G$ belongs to either $H$ or $\ov H$,
    and hence $c_G(u)$ differs from $c_G(v)$ is at least one of the coordinates.
    This implies that
    $\frac{\chi(H) + \chi(\ov{H})}{2} \geq \sqrt{\chi(H) \cdot \chi(\ov{H})} \geq \sqrt{k}$
    for all subgraphs $H \seq G$.
    Taking expectation of the inequality above, and recalling that
    each of $H$ and $\ov H$ are distributed like $G_{1/2}$
    we get that
    \[
        \E[\chi(G_{1/2})] = \frac{\E[\chi(H)] + \E[\chi(\ov{H})]}{2}
        \geq \sqrt{k}. \qedhere
    \]
\end{proof}
Next, we claim that $\chi(G_{1/2})$ is concentrated around its expectation.

\begin{lemma}\label{lemma:martingale}
    Let $G = (V,E)$ be a graph with $\chi(G) = k$. Then
    for all $0 < \eps < 1/2$ we have $\Pr[\chi(G_{1/2}) < (1- \eps) \cdot \E[\chi(G_{1/2})]]
    < e^{-\Omega(\eps^2 \cdot \E[\chi(G_{1/2})])}$.
\end{lemma}

\begin{proof}
    Let $V = C_1 \cup \dots \cup C_k$ be a partition of
    the vertices of $G$ into $k$ color classes, i.e., each $C_i$ is an independent set in $G$.
    Let $X_0, X_1, \dots, X_k$ be a sequence of random variables
    defined by
    \[
        X_i = \E \left[ \chi(G_{1/2}) \bigg| G_{1/2}[\cup_{j=1}^i C_j] \right].
    \]
    In words, the random variable $X_i$ samples random edged induced by $C_1 \cup \dots \cup C_i$,
    and then takes the expected chromatic number of the random subgraph $G_{1/2}$
    given the random edges chosen from $C_1 \cup \dots \cup C_i$ that have been already exposed.
    In particular $X_0 = \E[\chi(G)_{1/2})]$ and $X_k = \chi(G)_{1/2}$.
    Note that $X_{i+1} - X_i \in \{0,1\}$ since each $C_i$ is an independent set in $G$.

    We now use the following large deviation result due to Alon~et~al.~\cite{AGGL},
    which may be thought of as a multiplicative version of Azuma's inequality.
    \begin{proposition}\label{prop:AGGL}
    Let $X_1, \dots, X_k$ sequence of random variables adapted to some filter $(\F_i)$
    such that $|X_{i+1} - X_i| \leq 1$ for all $i=1,\dots, k-1$.
    Then for any $0 <\eps < 1/2$ it holds that
    \[
        \Pr \left[ \left| \frac{X_k}{\E[X_k]} - 1 \right| > \eps \right]
            < O(e^{-\Omega(\eps^2 \E[X_k])}) .
    \]
    \end{proposition}
    \medskip
    \noindent
    Therefore, using \cref{prop:AGGL} we have
    \[
        \Pr[\chi(G_{1/2}) < (1- \eps) \cdot \E[\chi(G_{1/2})]]
            = \Pr \left[ \frac{X_k}{\E[X_k]} < 1-\eps \right]
            < O(e^{-\Omega(\eps^2 \E[X_k])}).
    \]
\end{proof}
\cref{item:martingale} of \cref{thm:tiny chi} follows immediately from \cref{obs:sqrt(n)} and \cref{lemma:martingale}.
Indeed, by \cref{obs:sqrt(n)} we have $\E[\chi(G_{1/2}] \geq \sqrt{k}$, and by applying \cref{lemma:martingale}
we get that for all $0 < \eps < 1/2$ it holds that
\[
    \Pr[\chi(G_{1/2}) < (1- \eps) \cdot \sqrt{k}] \leq
    \Pr[\chi(G_{1/2}) < (1- \eps) \cdot \E[\chi(G_{1/2}]]
    <  e^{-\Omega(\eps^2 \cdot \E[\chi(G_{1/2}])}
    \leq e^{-\Omega(\eps^2 \cdot \sqrt{k})},
\]
as required.
\end{proof}

\section{Proofs of \cref{thm:alpha bound} and \cref{cor:subgraph col}}

Below we prove \cref{thm:alpha bound}. We remark that a similar idea has appeared in~\cite{BRS}.

\begin{proof}[Proof of \cref{thm:alpha bound}] 
    In our setting we have $\alpha(G) \leq \frac{C}{k} n$, and hence, every set
    of $\ell > 2\alpha(G)$ vertices spans at least
    $\frac{\ell(\ell/\alpha(G)-1)}{2} > \frac{\ell^2}{4\alpha(G)}$ edges.
    In particular, for $\ell = n/d$ we get that
    every set of $\ell$ vertices spans at least
    $\frac{\ell^2}{4\alpha(G)} \geq \frac{k \ell}{4 C d}$ edges.

    Let $S \seq V$ be a subset of the vertices of size $|S| = \ell = n/d$.
    Then, the probability that $S$ is an independent set in $G_{1/2}$
    is at most $(1-p)^{|E[S]|} \leq (1-p)^{\frac{k \ell}{4 C d}}$.
    Therefore, by taking union bound over all subsets of size $\ell$
    we get that
    \[
        \Pr[\chi(G_{1/2}) \leq d] \leq \Pr[\alpha(G_{1/2}) \geq \ell]
            \leq {n \choose \ell} \cdot (1-p)^{\frac{k \ell}{4 C d}}
            \leq \left( ed \right)^\ell \cdot e^{-\frac{p k \ell}{4 C d}}
            = e^{- (\frac{pk}{4 C d} - \ln(d)-1) \cdot \ell}.
    \]
    Note that if $d \leq \frac{pk}{16 C \ln(pk)}$, then $\frac{pk}{8 C d} > 2 \ln(pk) > \ln(d) + 1$,
    and hence $\frac{pk}{16 C d} -\ln(d) - 1 > \frac{pk}{8 C d}$.
    Therefore, the probability above is upper bounded by
    $2^{- \frac{k\ell}{8 C d}} = 2^{-\frac{p k n}{8 C d^2}}$,
    as required.

    \medskip\noindent
    In particular, by letting $d = \frac{pk}{16C \ln(pk)}$ we get
    \[
    \E[\chi(G_{1/2})] \geq d \Pr[\chi(G_{1/2}) \geq d]
    \geq d \cdot(1 - 2^{-\frac{p k n}{8 C d^2}}) = d \cdot (1 - 2^{-\frac{2 \ln(k) n}{d}}) \geq d/2 =  \frac{pk}{32C \log(pk)},
    \]
    and the theorem follows.

\end{proof}

Next we turn to proving \cref{cor:subgraph col}.
\begin{proof}[Proof of \cref{cor:subgraph col}]
    Let $G'$ be the subgraph as in the assumption.
    By \cref{thm:alpha bound} we have
    \[
        \E[\chi(G_{1/2})] > \E[\chi(G'_{1/2})] > \frac{k}{8C \log(k)},
    \]
    as required.
\end{proof}

\section{Proof of \cref{thm:hadwiger}}
The theorem follows almost immediately from the following result of Kostochka~\cite{Kostochka84}.
  \begin{theorem}[\cite{Kostochka84} Theorem 1]\label{thm:kostochka}
    Let $G = (V,E)$ be a graph such that $|E| \geq k \cdot |V|$.
    Then $h(G) \geq \Omega \left( \frac{k}{\sqrt{\log(k)}} \right)$.
  \end{theorem}

\begin{proof}[Proof of \cref{thm:hadwiger}]
  Suppose that $k$ is sufficiently large (e.g., $k \geq 10$), as otherwise the theorem holds trivially.
  Let $G$ be an $n$ vertex graph with $\chi(G) = k$.
  Let $G' = (V',E')$ be a $k$-critical subgraph of $G$, i.e., $G'$ is a subgraph of $G$ such that $\chi(G') = k$
  but removing any edge from $G'$ reduces its chromatic number.
  Then, every vertex of $G'$ has degree at least $k-1$,
  and hence, $|V'| \geq k$ and $|E'| \geq \frac{k-1}{2}|V'| \geq \frac{k}{4}|V'|$.
  Let $H = (V',E_H) \sim G'_{1/2}$. Then, by Chernoff bound we have
  $\Pr[|E_H| < k \cdot |V'|/8] < \exp(-\Omega(k \cdot |V'|)) < \exp \left(-\Omega(k^2) \right)$.
  Applying \cref{thm:kostochka} to $H$ we get that $h(H) \geq \Omega \left( \frac{k}{\sqrt{\log(k)}} \right)$ with probability $1 - \exp \left(-\Omega(k^2) \right)$.
  This completes the proof of \cref{thm:hadwiger}.
\end{proof}

\section{Open problems}\label{sec:open}

The most obvious open problem in this context is the original question of Bukh.
\begin{question}\label{prob:subgraphs}
    Is there a constant $c>0$ such that $\E[\chi(G_{1/2})]>c \cdot \frac{\chi(G)}{\log \chi(G)}$ for all $G$?
\end{question}

Other than Bukh's original question, this paper raises several additional problems which we mention below.

\begin{question}\label{prob:subgraphs}
  Is it true that every graph $G$ contains an induced subgraph $G' \seq G$
  such that $\chi(G') \geq c \cdot \chi(G)$, and $\alpha(G') \leq C \frac{|V(G')|}{\chi(G')}$
  for some absolute constants $C,c>0$?
\end{question}

A positive answer to this question would immediately give a positive answer to Bukh's question using \cref{thm:alpha bound}.
We stress that \cref{prob:subgraphs} does not require any conditions on the number of vertices on $G'$, except for the obvious $|V(G')| \geq \chi(G') \cdot \alpha(G')/C$.

\medskip
In this paper we focused on the chromatic number of $G_p$ for $p = 1/2$.
It would be interesting to extend the current results to other values of $p$.
In particular, it would be interesting to answer \cref{prob:domination} for all values of $p \in (0,1)$.
\begin{question}\label{prob:domination}
  Let $p \in (0,1)$. Is there a constant $C > 1$ and a polynomial $\poly : \R \to \R$ such that
  $\Pr[\chi(G_{p}) \leq d] < \poly(\Pr[\chi(G(k,p)) \leq C \cdot d])$
  holds for every graph $G$, and for all $d \leq k$,
  where $k = \chi(G)$ and $G(k,p)$ is the random Erd\"{o}s-R\'{e}nyi graph model?
\end{question}
\noindent

We conclude with two more problems that we find interesting.
\begin{question}\label{prob:p>1/2}
    Is it true that $\E[\chi(G_{p})] \geq \chi(G)^{p}$ for all $G$ and all $p > 1/2$?
\end{question}

\begin{question}\label{prob:p/2}
    Let $p \in (0,1)$. Is there a constant $c = c(p) > 0$
    such that $\E[\chi(G_{p/2})] \geq c \cdot \E[\chi(G_{p})]$ for all $G$?
\end{question}

\section{Acknowledgements}\label{sec:ack}
I am grateful to Huck Bennett for many helpful discussions related to this work.
Huck made several crucial observations related to this work, but persistently refused to co-author the paper.
I am also thankful to Uriel Feige and Daniel Reichman for very useful discussions related to this paper.

\bibliographystyle{alpha}
\bibliography{percCol}

\end{document}